\documentclass[12pt]{article}

\usepackage{graphicx}
\usepackage{epstopdf}
\usepackage{amsmath, amssymb}
\usepackage{amsthm} 

\theoremstyle{plain} 
\newtheorem{theorem}{\indent\sc Theorem}[section]

\newtheorem{corollary}[theorem]{\indent\sc Corollary}
\newtheorem{proposition}[theorem]{\indent\sc Proposition}

\theoremstyle{definition} 
\newtheorem{definition}[theorem]{\indent\sc Definition}
\newtheorem{remark}[theorem]{\indent\sc Remark}

\makeatother

\title{\uppercase {On the Prolongations of Representations of Lie Groups}}

\author{Hulya Kadioglu\\Department of Mathematics, Gazi University, Turkey\\and Department of Mathematics, Idaho State University, USA\\emails: hulyakaya@gazi.edu.tr and kayahuly@isu.edu\\Erdogan Esin \\Department of Mathematics, Gazi University, Turkey.\\email: eresin@gazi.edu.tr 
}
\date{}
\begin{document}
   \maketitle

   \begin{abstract}
\noindent In this paper, we introduce a study of  prolongations of representations of Lie 
groups. We obtain a faithful ({\em one-to-one}) representation of $TG$ where $G$
  is a finite-dimensional Lie group and $TG$ is the tangent bundle of  $G$, by 
using (not necessarily faithful) representations  of $G$. We show that tangent
 functions of  Lie group actions correspond to prolonged representations. We 
prove that if two representations are equivalent, then their prolongations are 
equivalent too. We show that if $U$ is an invariant subspace for a 
representation, then $TU$ is an invariant subspace for the prolongation of
the  given representation and vice versa. We prove that if $\widetilde{\Phi}$
 is an irreducible representation, then $\Phi$ is also an irreducible 
representation. Finally we show that prolongations of direct sum of two 
representations are direct sum of their  prolongations.
  \newline
  
  \noindent \textbf{Keywords:} Prolongation; Lie Group; Representation on Lie Group; Tangent Bundle. 

    \end{abstract}

\section {Introduction}

\noindent In this paper, we present a study of  prolongations of representations of Lie
 groups. A representation denoted by $(G,V)$ on a Lie group $G$ can be defined 
as a homomorphism from $G$ to a automorphism group of a finite-dimensional 
vector space $V$.  
\newline

\noindent It is known that if $\Phi:G\to Aut(V)$ is a {\em one-to-one} (faithful) 
representation, then $\Phi(G)$ is a matrix
group  and isomorphic to the original group $G$ \cite{Hall}. This helps us 
to represent $G$ as a matrix Lie group \cite{Hall}. The main benefit of such representations 
is that it is 
easier to execute algebraic operations/computations on a matrix Lie group 
than a standard Lie group. 
In general, a prolongation prolongs structures of  manifolds to their bundles.
 In this paper, we are interested in studying prolongations of representations
 of Lie groups to their tangent bundles.
 \newline

\noindent In 1966, K. Yano and S.Kobayashi proposed the following question:
 Is it possible to associate each $G$-structure on a smooth manifold $M$ with a naturally
 induced $G'$-structure on the tangent bundle $TM$ where $G'$ is a certain subgroup of 
$GL(2n,\mathbb{R})$ \cite{Yano1}? In 1968, Morimoto gave an
 answer to this question \cite{Morimoto:1968}. In \cite{Morimoto:1968}, first a vector
 space structure with a new sum and scalar product on $\ T\mathbb{R}^n$ 
 (the tangent bundle of $\mathbb{R}^n$)  is introduced and generalized to an 
arbitrary finite-dimensional vector space V \cite{Morimoto:1968}. Then
 $T(GL(n,\mathbb{R}))$ was embedded
 into $GL(2n,\mathbb{R})$ by using the following Lie group homomorphism  
   \begin{equation*}
    J_n:T(GL(n,\mathbb{R}))\to GL(2n,\mathbb{R}).
   \end{equation*}
 Using this Lie group homomorphism, Morimoto finds an association between each $G$-structure on a smooth manifold $M$ and a naturally
 induced $J_n(TG)$-structure on $TM$. Although our research is not direct result of Morimoto's study, his work provided us important insights such as using the vector space structure on $\ T\mathbb{R}^n$ and the Lie group homomorphism $J_n$ leaded us to following important findings: We show that the bundle trivialization of $TV$ is a linear
 isomorphism from $TV$ to $V\times \mathbb{R}^n$. Then using this isomorphism, we 
obtain a basis for $TV$. Based on these, we define the following  new 
{\em one-to-one} representation 
    \begin{equation*}
    \widetilde{\Phi}:TG\to Aut(TV), 
    \end{equation*} 
where $\Phi:G \to Aut(V)$ is a finite-dimensional real representation  and
 $\widetilde{\Phi}$ is a prolongation of the
 representation $\Phi$ to the tangent bundle $TG$ (refer to 
Section \ref{original}). In addition, we show that if $\rho$ is a group action 
corresponding to the representation $\Phi$, then $T\rho$ is a group  action
 that corresponds to the prolonged representation $\widetilde{\Phi}$. 
Using these, we 
prove that if two representations are equivalent, then their prolongations are
 also equivalent. We show that if $U$ is an invariant subspace for $\Phi$, 
then $TU$ is an invariant subspace for $\widetilde{\Phi}$ and vice versa. We prove
 that if $\widetilde{\Phi}$ is an irreducible representation, then $\Phi$ is 
also an irreducible representation. We note that if  $\Phi$ is an irreducible 
representation, then $\widetilde{\Phi}$ is not necessarily an irreducible 
representation (refer to Section \ref{original}). Finally, we show that 
prolongations of direct sum of two representations are direct sum of
 their prolongations.




\section{Background} 
In this section, we present the following basic definitions and theorems that will be used in Section \ref{original}.\\
\begin{theorem}\label{teoteo1}
For manifolds $\ M$ and $\ N$, $\ T(M)\times T(N)$ is equivalent to $\ T(M\times N)$ by using the following relation 
    \begin{equation}
    \begin{gathered}
    \numberwithin{equation}{section}
     (X,Y)\cong Tf_x(Y)+T\bar{f}_y(X) 
    \end{gathered}
    \end{equation}

for all $X\in T_x(M)$ and $Y\in T_y(M)$, where $f_x:N\to M\times N$ and $\bar{f}_y:M\to M\times N$ defined by $\ f_x(m)=(x,m)$ and 
$\bar{f}_y(m)=(m,y)$, where $T_x(M)$ represents the tangent space of $M$ at $\ x\in{M}$. 
\end{theorem}
\begin{definition}
If we consider a coordinate neighborhood$\ U$ in$\ M$ with a local coordinate system $\{x_1,x_2,...,x_n\}$, then we can canonically define a local coordinate system\\ $\{x_1,x_2,...,x_n,v_1,v_2,...,v_n\} $ on $ T(U)$, i.e., a tangent vector $\displaystyle\sum_{i=1}^n {v_i (\frac{\partial}{\partial x_i})_x}$ has the coordinates $\ (x_1,x_2,...,x_n,v_1,v_2,...,v_n)$ if the point $\ x\in{U}$ has the coordinates $ (x_1,x_2,...,x_n)$. This local coordinate system $\{x_1,x_2,...,x_n,v_1,v_2,...,v_n\}$ is called the induced local coordinate system on $\ T(U)$ by $\{x_1,x_2,...,x_n\}$ \cite{Clark,Saunders}.

\end{definition}

\begin{definition}\label{def1} 
If we consider two tangent vectors $X\in T_x(\mathbb{R}^n)$
 and $Y\in T_y(\mathbb{R}^n)$, then the tangent bundle $T(\mathbb{R}^n)$ is a
 vector space of dimension $2n$ with respect to the following sum "$\oplus $" 
and the scalar multiplication "$\bullet$"
\begin{equation}
\begin{gathered}
 X\oplus Y=(T\tau_y)X+(T\tau_x)Y, \\
 \lambda\bullet X=(T\sigma_{\lambda})X. 
\end{gathered}
\end{equation}
where $\tau_x$ represents a translation of $\mathbb{R}^n$ by $ x\in \mathbb{R}^n$ and $\sigma_{\lambda}$ represents a the scalar multiplication by $\lambda \in \mathbb{R}$.
\noindent For any finite-dimensional vector space $V$, the tangent bundle $T(V)$ 
becomes a vector space with respect to the similar sum and the scalar
 multiplication. If $\ f:\mathbb{R}^n\to \mathbb{R}^m$ is a linear map, $\ Tf:T(\mathbb{R}^n)\to T(\mathbb{R}^m)$ (differential of $f$) is also a linear map  \cite{Morimoto:1968}.
\end{definition}

\begin{definition}\label{def2} Let $R_a:GL(n)\to GL(n)$ be a right translation of $ GL(n)$ by $ a\in GL(n)$ where
 $R_a(y)=y.a$ for $y\in GL(n)$. 
Then $B=TR_{a^{-1}}(Y)$ is a tangent vector of $GL(n)$ at its unit element ($e \in GL(n)$), 
namely $B \in Lie(GL(n))$.\\

\noindent Conversely, for any pair $a\in GL(n)$ and $ B\in Lie(GL(n))$, there exists 
 
\noindent $Y\in T_a(GL(n))$ by $Y=TR_a(B)$.  $Y$ can also be written as $Y=[a,B]$ \cite{Morimoto:1968}.
\end{definition}

\begin{theorem} \label{teo3}
$\ T(GL(n,\mathbb{R}))$ can be embedded into $\ GL(2n,\mathbb{R})$ by the following {\em{one-to-one}} 
Lie group homomorphism 
  \begin{equation}
  \begin{gathered}
  j_n:T(GL(n,\mathbb{R}))\to GL(2n,\mathbb{R}),\\ 
  J_n([a,B])=\
 \begin{pmatrix}
a & 0\\
Ba & a
\end{pmatrix}.
\end{gathered}
\label{jn}
\end{equation}
for any $a\in GL(n)$ and $B\in Lie(GL(n))$. It can be shown that $\ J_n([a,0])=Ta$ \cite{Morimoto:1968}.
\end{theorem}

\begin{remark}\label{remark4}
  The matrix that corresponds to
 a linear operator $\ F\in Aut(V)$ with respect to a fixed basis,
 $\{\alpha_i : 1\le i\le n\}$ $\subset$ $V$, consists of  arrays of scalars $\ (F_{j}^{i})$ determined by   
    \begin{equation}
    \begin{gathered}
    F(\alpha_j)=\displaystyle\sum_{i=1}^n (F_{j}^{i})\alpha_i.
    \end{gathered}
    \label{rem24}
    \end{equation}
\cite{Bel}. Here, $ F_{j}^{i}$ represents $\ (i,j)^{th}$ entry of the matrix that corresponds to the linear map $F$. Using (\ref{rem24}),  we can define the following  group isomorphism 
    \begin{equation}
    \begin{gathered}
     Z:Aut(V)\to GL(n)\hspace{4mm}, Z(F)=(F_{j}^{i}).
     \end{gathered}
     \end{equation}
Moreover, $F$ can be written as $F=F_{j}^{i} \alpha_i \otimes \alpha_j$ where $\ F_{j}^{i}\in GL(n)$. 
Another group isomorphism $\breve{Z}$ from $Aut(TV)$ to  $GL(2n)$ can be defined similar fashion. 
\end{remark}

\begin{proposition}\label{prop5} Let $ G$ and $ G'$ be two Lie groups. If $\gamma:G\to G'$
 is a Lie group homomorphism, then $T\gamma:TG\to TG'$ is a {\em one-to-one} Lie group
 homomorphism. 
\end{proposition}

\begin{remark}\label{remark6} Let $\ (\bar{y}_{j}^{i}):Aut(V)\to \mathbb{R}^{n^2}$ and $\ (y_{j}^{i}):GL(n,\mathbb{R})\to \mathbb{R}^{n^2}$ be coordinate functions of $ Aut(V)$ and $\ GL(n,\mathbb{R})$ respectively, then the coordinate representative of $\ Z$ is $\ I_{\mathbb{R}^{n^2}}$. Thus $\ (\bar{y}_{j}^{i})= ({y}_{j}^{i})\circ Z$. 
\end{remark}

\section{ Prolongations of Representations of Lie Groups}\label{original}

\begin{proposition}\label{newpro1} Let $V$ and $W$ be arbitrary finite-dimensional real vector spaces of 
dimensions $ n$, $ m$ and $ f:V\to W$ be a linear map, then the tangential map $ Tf:TV\to TW$ is a 
linear function. Moreover if $ f$ is a linear isomorphism, then $ Tf$ is also a linear
 isomorphism.
\end{proposition}

\begin{proof} The proof is similar to \cite{Morimoto:1968}.
\end{proof}

\begin{proposition}\label{newpro2} Let $\psi:TV\to V\times \mathbb{R}^n$ be the bundle trivialization of the tangent bundle $\ TV$. Then $\psi$ is a linear isomorphism with respect to the both vector space structure on $\ V\times \mathbb{R}^n$ and the structure on $\ TV$ defined in Definition\ref{def1}.
\end{proposition}

\begin{proof} Since $\psi$ is a bundle trivialization, it is by definition 
{\em{one-to-one}} and {\em{onto}}. Therefore, showing that $\psi$ is a linear function 
will complete the proof. 

\noindent Let $v=v_i (\frac{\partial}{\partial x_i})|_{p_1}$ and
 $w=w_i (\frac{\partial}{\partial x_i})|_{p_2}$ be arbitrary elements of $\ TV$. 
 For all $r \in \mathbb{R}^n$, 
\begin{equation}
(x_i\circ \tau_{p_1}\circ \xi^{-1})(r)=p_{1}^{i}+r_i,
\end{equation}

\noindent where  $x_i:V\to \mathbb{R}$ ($ 1\le i\le n$) are a global coordinate functions with 
$x_i=u_i\circ \xi$, $u_i$ ($ 1\le i\le n$) are standard coordinate functions of $\mathbb{R}^n$,  $\xi$ is a canonical linear isomorphism
 from $V$ to $\mathbb{R}^n$, $p_{1}^{i} = x_i(p_1)$, and $r_i \in \mathbb{R}$.
 Since $p_{1}^{i}$'s are constants for each $\ i\in \{1,2,....,n\}$, we have 
\begin{equation}
\begin{gathered}
\frac{\partial(x_i\circ \tau_{p_1})}{\partial x_j}|_{p_2}=\frac{\partial(x_i\circ \tau_{p_1}\circ \xi^{-1})}{\partial u_j}|_{\xi(p_2)}=\delta_{ij}.
\end{gathered}
\end{equation}
 Thus 
\begin{equation}
\begin{gathered} 
T\tau_{p_2}(v)[x_i]=\frac{\partial (x_i\circ \tau_{p_2})}{\partial x_j}|_{p_1} v_j=v_i.
\end{gathered}
\label{thus33}
\end{equation}
 Similarly  
 \begin{equation}
\begin{gathered} 
  T\tau_{p_1}(w)[x_i]=w_i.\label{sum} 
\end{gathered}
\end{equation} 
 On the other hand, using $\ (x_i\circ \sigma_c\circ \xi^{-1})(r)=c r_i$ with 
$c \in \mathbb{R}$, we get  
\begin{equation}
\begin{gathered} 
\frac{\partial(x_i\circ \sigma_c)}{\partial x_j}|_{p_2}=\frac{\partial(x_i\circ \sigma_c\circ \xi^{-1})}{\partial u_j}|_{\xi(p_2)}=c\delta_{ij}\label{theo}.
\end{gathered}
\end{equation} 
Using (\ref{theo}) we get
\begin{equation}
\begin{gathered} 
(T\sigma_c(v))[x_i]=\frac{\partial (x_i\circ \sigma_c)}{\partial x_j}|_{p_1} v_j=cv_i.\label{scalar}
\end{gathered}
\end{equation} 
 
\noindent Using  (\ref{thus33}) and (\ref{sum}) we get
\begin{eqnarray}
\psi(v\oplus w)&=&\psi(T\tau_{p_2}(v)+T\tau_{p_1}(w)), \nonumber \\
&=&(p_1+p_2,(T\tau_{p_2}(v)+T\tau_{p_1}(w))[x_i]e_i), \nonumber\\
&=&(p_1+p_2,(v_i+w_i)e_i), \nonumber \\
&=&(p_1,v_i e_i)+(p_2,w_ie_i), \nonumber \\
&=&\psi(v)+\psi(w)\label{ax1}.
\end{eqnarray} 

\noindent  Furthermore, using (\ref{scalar}) we have
 \begin{eqnarray}
 \psi(c\bullet v)&=&(cp_1,T\sigma_c(v)[x_i]e_i), \nonumber \\
 &=&(cp_1,v[x_i\circ \sigma_c]e_i), \nonumber \\
 &=&(cp_1,cv_ie_i), \nonumber \\
 &=&c(p_1,v_ie_i), \nonumber \\ 
 &=&c\psi(v)\label{ax2}.
 \end{eqnarray}

\noindent  (\ref{ax1}) and (\ref{ax2}) indicate that  $\psi$ is a linear function.
 This ends the proof.  
\end{proof}

\begin{corollary}\label{newcor3} If $\{\alpha_i : 1\le i\le n\}$ is a basis for $V$
 and $\{e_i :1\le i\le n\}$ is the standard basis for $\mathbb{R}^n$, then we can
 define a basis for $\ TV$ by using bundle trivialization $\psi$. 
\end{corollary}

\begin{proof} Let us define $\bar{\alpha_i}=(\alpha_i,0)\in V\times 
\mathbb{R}^n$ and $\ y_i=(0,e_i)\in V\times \mathbb{R}^n$ for $\forall{i}
 \in \{1,2,...,n\}$.
Then by definition  $\eta=\{\bar{\alpha_i},y_i : 1\le i\le n\}$ is a basis for 
$V\times \mathbb{R}^n$. Since $\psi$ is a linear isomorphism, 
then $$\psi^{-1}(\eta) =\{\widetilde{\alpha_i},\widetilde{y_i} : 
\widetilde{\alpha_i}= \psi^{-1}(\bar{\alpha_i}), \widetilde{y_i}=\psi^{-1}(y_i) \}$$
 is a basis for $\ TV$.
\end{proof}

\begin{proposition}\label{newpro4} If $ (G,V)$ is an $ n$-dimensional real 
representation, then $(TG,TV)$ is a $2n$-dimensional real faithful representation.
\end{proposition}

\begin{proof} Let $\Phi = (G,V) : G\to Aut(V)$ be an $n$-dimensional real
 representation, $J_n$ be a {\em one-to-one}  homomorphism defined by equation
 (\ref{jn}), and $ Z$,  $\breve{Z}$ be  isomorphisms defined in Remark 2.4.
From Proposition(\ref{prop5}). $T \Phi$ is  a {\em one-to-one} homomorphism. Using this,
we can define a {\em one-to-one} homomorphism as  
\begin{eqnarray}
\widetilde{\Phi} &=& (TG,TV) : TG\to Aut(TV),\nonumber \\
\widetilde{\Phi} &=& \breve{Z}^{-1}\circ J_n\circ TZ\circ T\Phi.
\label{propro1}
\end{eqnarray}

\noindent Since $TV$ is a $2n$-dimensional vector space, then $ (TG,TV)$ is a $2n$-dimensional 
  faithful representation. 
\end{proof}

\begin{definition}\label{newdef5} The representation  $\widetilde{\Phi}$ defined in the 
proof 3.4 is referred to as  the prolongation of $\Phi$ to $TG$. 
\end{definition}

\begin{remark}\label{newremark6} For the representation  $\Phi$ and its prolongation 
$\widetilde{\Phi}$ (Eqn. \ref{propro1}), we have the following 

\begin{eqnarray}
\widetilde{\Phi}(X_a)=((Z\circ \Phi)(a))_{j}^{i} \widetilde{\alpha}_i\otimes \widetilde{\alpha}_j^\ast +(TZ(B) (Z\circ \Phi)(a))_{j}^{i} \widetilde{y}_i \otimes \widetilde{\alpha}_j^\ast \nonumber \\
 + ((Z\circ \phi)(a))_{j}^{i} \widetilde{y}_i \otimes \widetilde{y}_j^\ast, 
\label{propro2}
\end{eqnarray}

\begin{eqnarray}
(\widetilde{\Phi}(X_a))(Y_p)= (((Z\circ \Phi)(a))_{j}^{i} p_j \alpha_i ,(TZ(B) (Z\circ \Phi)(a))_{j}^{i} p_j e_i \nonumber \\
+(Z\circ \Phi)(a))_{j}^{i}Y_j e_i),
\label{propro3}
\end{eqnarray}

\noindent where $X_a\in TG$,  $B=TR_{\Phi(a)}^{-1}(T \Phi(X_a))$, and $ Y_p\in TV$.
\end{remark}

\begin{proposition}\label{newpro7} Let $\rho$ be a group action for $(G,V)$.
 Then $T\rho$ is a group action  that corresponds to   $(TG,TV)$.
\end{proposition}

\begin{proof} Since $\rho$ is a group action  for  $ (G,V)$,   we have 
 $\rho:G\times V\to V$ with  $\rho(a,p)=\Phi(a)(p)$ for all $ (a,p)\in G\times V$. 
Now, let us define the following mappings

\begin{equation}
\begin{gathered}
 x_j:V\to \mathbb{R},\hspace{3mm}x_j(p)=p_j, \\ 
 \bar{x}_j:G\times V\to\mathbb{R},\hspace{3mm}\bar{x}_j(a,p)=p_j,\\ 
 \bar{y}_{j}^{i}:Aut(V)\to \mathbb{R},\hspace{3mm} \bar{y}_{j}^{i}(f)=f_{j}^{i},\\ 
 \widetilde{y}_{j}^{i}:G\times V\to \mathbb{R},\hspace{3mm} \widetilde{y}_{j}^{i}(a,p)=(\Phi(a))_{j}^{i}.
 \end{gathered} 
\label{mapmap1}
 \end{equation} 
 
 \noindent Using (\ref{mapmap1})  we get  
$(x_j\circ \rho)=\displaystyle\sum_{t=1}^n \widetilde{y}_{j}^{t}.\bar{x}_t$. 
From this we have 
\begin{eqnarray}
T\rho (X_a,Y_p)[x_j]&=&(X_a,Y_p)[x_j\circ \rho], \nonumber \\ 
 &=&(X_a,Y_p)(\displaystyle\sum_{t=1}^n \widetilde{y}_{t}^{j}.\bar{x}_t), \nonumber \\ 
 &=&\displaystyle\sum_{t=1}^n (X_a,Y_p)(\widetilde{y}_{t}^{j})(\bar{x}_t(a,p))+ \displaystyle\sum_{t=1}^n \widetilde{y}_{t}^{j}(a,p)(X_a,Y_p)(\bar{x}_t), \nonumber\\ 
 &=&\displaystyle\sum_{t=1}^n X_a[\widetilde{y}_{t}^{j}\circ \bar{f}_p]p_t+\displaystyle\sum_{t=1}^n(\Phi(a))_{t}^{j} Y_p[\bar{x}_t\circ f_a], \nonumber  \\
 &=&\displaystyle\sum_{t=1}^n X_a[\bar{y}_{t}^{j}\circ \Phi]p_t+\displaystyle\sum_{t=1}^n(\Phi(a))_{t}^{j} Y_p[x_t], \nonumber  \\
 &=&\displaystyle\sum_{t=1}^n T\Phi(X_a)[\bar{y}_{t}^{j}]p_t+\displaystyle\sum_{t=1}^n(\Phi(a))_{t}^{j} Y_t .
\end{eqnarray}

\noindent Since $T\rho(X_a,Y_p)$ is a tangent vector at $\rho(a,p)$, then we can write 
\begin{gather}
 T\rho(X_a,Y_p)=(\rho(a,p),T\Phi(X_a)[\bar{y}_{t}^{j}]p_te_j+(\Phi(a)) _{t}^{j} Y_t e_j)\label{ax3}.
\end{gather}

\noindent Using  the right translation of $\Phi(a)$  and (\ref{mapmap1}), we have 
\begin{equation}
(\bar{y}_{t}^{j}\circ R_{\Phi(a)})(f)=\displaystyle\sum_{k=1}^n \bar{y}_{k}^{j}(f)\Phi(a)_{t}^{k},
\label{rtrans1}
\end{equation}

\noindent where $f \in Aut(V)$.  Taking the partial derivative of  (\ref{rtrans1}), we
 get
     
\begin{equation}
\frac{\partial(\bar{y}_{t}^{j}\circ R_{\Phi(a)})}{\partial \bar{y}_{l}^{q}}=\delta_{q}^{j}(\Phi(a))_{t}^{l}.
\label{rtrans2}
\end{equation}

\noindent Now let us define $T\Phi (X_a)=TR_{\Phi(a)}(B)$ with $\ B=\displaystyle\sum_{q,l=1}^n b_{l}^{q}\frac{\partial}{\partial \bar{y}_{l}^{q}}|_I$. Using this and 
 equation (\ref{rtrans2}),  we get
\begin{eqnarray}
 T\Phi(X_a)[\bar{y}_{t}^{j}]&=&(TR_{\Phi(a)}B)[\bar{y}_{t}^{j}], \nonumber \\
                            &=&\displaystyle\sum_{q,l=1}^n b_{l}^{q}\frac{\partial(\bar{y}_{t}^{j}\circ R_{\Phi(a)})}{\partial \bar{y}_{l}^{q}}|_I,\nonumber \\                                &=&\displaystyle\sum_{k=1}^n b_{k}^{j}\Phi(a)_{t}^{k}, \nonumber \\
                            &=&\displaystyle\sum_{k=1}^n (TZ(B))_{k}^{j}((Z\circ \Phi)(a))_{t}^{k}, \nonumber \\
                            &=&((TZ(B) (Z\circ \Phi)(a))_{t}^{j}. \label{eqn}
\end{eqnarray}
\noindent If we rewrite (\ref{ax3}) using (\ref{eqn}), then we obtain
\begin{eqnarray}
T\rho (X_a,Y_p)&=&(\rho (a,p),((TZ(B) (Z\circ \Phi)(a))_{t}^{j}p_te_j+(\Phi(a)) _{t}^{j} Y_t e_j), \nonumber \\
               &=&(\Phi(a)(p),((TZ(B) (Z\circ \Phi)(a))_{t}^{j}p_te_j+(\Phi(a)) _{t}^{j} Y_t e_j), \nonumber \\
               &=&(((Z\circ \Phi)(a))_{j}^{i}p_j\alpha_i,((TZ(B) (Z\circ \Phi)(a))_{t}^{j}p_te_j+(\Phi(a)) _{t}^{j} Y_t e_j),\nonumber \\
               &=&\widetilde{\Phi}(X_a)(Y_p) \label{lasteqn}
\end{eqnarray}
that shows  that $T\rho$ is a group action for $\widetilde{\Phi} = (TG,TV)$. 
\end{proof}

\begin{proposition}\label{newpro8} If $(G,V)$ and $ (G,V')$ are two equivalent representations, then their prolongations $(TG,TV)$ and $\ (TG,TV')$ are 
 equivalent too.
\end{proposition}

\begin{proof} Let $(G,V)$ and $(G,V')$ represent group homomorphisms $\Phi:G\to Aut(V)$ and $\Phi':G\to Aut(V')$ together with the  corresponding group  actions 
 $\rho$ and $\rho'$. Since  $ (G,V)$ and $ (G,V')$ are equivalent, there exists a linear isomorphism $ A:V\to V'$ such that $ A(\Phi(a)(p))=\Phi'(a)(A(p))$  for all $ a\in G$
 and $ p\in V$. Then we have 
\begin{eqnarray}
A\circ \rho&=&\rho'\circ (I\times A),\nonumber \\ 
TA\circ T\rho&=&T\rho' \circ (TI\times TA),\nonumber \\
&=&T\rho' \circ (I\times TA).\label{rho}
\end{eqnarray}
Using equations (\ref{rho}) and (\ref{ax3}), we get
\begin{eqnarray}
       (TA\circ T\rho)(X_a,Y_p)&=&T\rho' \circ (TI\times TA)(X_a,Y_p),\nonumber \\ 
            TA(T\rho (X_a,Y_p))&=&T\rho'(X_a,TA(Y_p)),\nonumber \\
(TA(\widetilde{\Phi}(X_a))(Y_p)&=&(\widetilde{\Phi'}(X_a))(TA(Y_p))\label{equi} 
\end{eqnarray}

\noindent for all $ (X_a,Y_p)\in TG\times TV$. Since $TA:TV\to TV'$ is a linear isomorphism from Proposition(\ref{newpro1}). and satisfies (\ref{equi}), then by 
definition  $(TG,TV)$ and $ (TG,TV')$ are equivalent.
\end{proof}

\begin{proposition}\label{newpro9} $U$ is an invariant subspace for $\Phi$ if and only if $TU$  is an invariant subspace for  $\widetilde{\Phi}$.\\ 
\end{proposition}

\begin{proof} Let $U$ be a $k$-dimensional invariant subspace for $\Phi$. This means that  for all $p\in U$,   $\Phi(a)(p)\in U$.
 On the other hand, since $\widetilde{\Phi}(X_a)$ is a linear isomorphism for all $X_a\in TG$, then $ dim((\widetilde{\Phi}(X_a))(TU))=dim TU=2k$. 
Therefore $(\widetilde{\Phi}(X_a))(Y_p)\in TU$ for all $Y_p\in TU$  showing that  $TU$ is an invariant subspace for $\widetilde{\Phi}$.
\newline
 
\noindent Conversely, let us  assume  that $TU$ is an invariant subspace for $\widetilde{\Phi}$. For all $a\in G$ and $p\in U$, there exists a tangent vector 
$Y_p\in TU$ and $X_a\in TG$. Since $TU$ is an invariant subspace for $\widetilde{\Phi}$, we have $\widetilde{\Phi}(X_a)(Y_p)\in TU$ and  
\begin{equation}
\widetilde{\Phi}(X_a)(Y_p) = (\Phi(a)(p),((TZ(B).(Z\circ \Phi)(a))_{t}^{j}p_te_j+(\Phi(a)) _{t}^{j}.Y_t e_j)\in U\times \mathbb{R}^k.\nonumber \\
\end{equation}
This indicates that  $\Phi(a)(p)\in U$ therefore $U$ is an invariant subspace for $\Phi$. 
\end{proof} 

\begin{corollary}\label{invariant} If $\widetilde{\Phi}$ is an  irreducible representation, then  $\Phi$ is an  irreducible representation too.\\ 
\end{corollary}

\begin{proof} Let $\widetilde{\Phi}$ be an irreducible representation and $W$ be an invariant subspace for $\Phi$.
 By Proposition(\ref{newpro9}), $TW$ is an invariant subspace for $\widetilde{\Phi}$. Since $\widetilde{\Phi}$ is irreducible, $TW=TV$ or $ TW=\{0\}$ which implies that 
$W=V$ or $W=\{0\}$. Therefore $\Phi$ is an irreducible representation.
\end{proof}

\noindent Notice that the converse is not true due to the following  counter example: Let us consider 
\begin{equation*}
\Phi:S^1 \to Aut(\mathbb{R}^2), \hspace{3mm} \\
\Phi(a,b)(x,y)=(ax+by,-bx+ay)
\end{equation*}
where $(a,b)\in S^1$ and $(x,y)\in \mathbb{R}^2$.
\noindent It can be easily shown that $\Phi$ is an irreducible representation. However, its prolongation $\widetilde{\Phi}$ is not an irreducible 
representation since $\{(0,0,x,y)| x,y \in \mathbb{R}\} \subset \mathbb{R}^4 = T\mathbb{R}^2$ is a nontrivial invariant subspace for $\widetilde{\Phi}$.
\newline

\begin{corollary}\label{newcor11} The prolongation of direct sum of two representations is  the direct sum of prolongations of those two representations.
\end{corollary}

\begin{proof} Let us consider two finite-dimensional real representations  $\Phi_1$, $\Phi_2$ together with corresponding group actions $\rho_1$,  $\rho_2$
and their direct sum $\Phi_1 \oplus \Phi_2$. Let $Pr_1$ and $Pr_2$ represent first and second Cartesian projection of $G\times (V_1 \oplus V_2)$ respectively i.e.,
 $Pr_1(a,(v_1,v_2))=a$ and $Pr_2(a,(v_1,v_2))=(v_1,v_2)$. Also let  $pr_1$ and $pr_2$ represent first and second Cartesian projections of $V_1 \oplus V_2$.
 Then we have 
\begin{eqnarray}
(\rho_1 \oplus \rho_2)(a,(v_1,v_2))&=&(\rho_1(a,v_1),\rho_2(a,v_2)), \nonumber \\
                                   &=&(\rho_1 \times \rho_2)((a,v_1),(a,v_2)), \nonumber \\
                                   &=&(\rho_1 \times \rho_2)\circ ((Pr_1, pr_1 \circ Pr_2), \nonumber \\
                                   & &(Pr_1, pr_2 \circ Pr_2))(a,(v_1,v_2))
\end{eqnarray}
for all $a\in G$ and $(v_1,v_2) \in V_1\oplus V_2$. This implies that 
\begin{equation}
(\rho_1 \oplus \rho_2)=(\rho_1 \times \rho_2)\circ ((Pr_1, pr_1 \circ Pr_2),(Pr_1, pr_2 \circ Pr_2))\label{rho11}.
\end{equation}
Using equation (\ref{rho11}) and chain rule, we have

\begin{eqnarray}
T(\rho_1 \oplus \rho_2)&=&T((\rho_1 \times \rho_2)\circ ((Pr_1, pr_1 \circ Pr_2),(Pr_1, pr_2 \circ Pr_2))), \nonumber \\
                       &=&(T(\rho_1)\times T(\rho_2)) \circ ((T(Pr_1), T(pr_1) \circ T(Pr_2)), \nonumber \\
                       & &(T(Pr_1), T(pr_2) \circ T(Pr_2))).\label{phiandrho}
\end{eqnarray}
Since the tangent functions of the first and second Cartesian product projections are also first and second
 Cartesian product projections and using (\ref{phiandrho}), we get
\begin{eqnarray}
\widetilde{(\Phi_1 \oplus \Phi_2)}(X)(Y_1,Y_2)&=&T(\rho_1 \oplus \rho_2)(X,(Y_1,Y_2)), \nonumber \\
                                              &=&(T(\rho_1)\times T(\rho_2))((X,Y_1),(X,Y_2)),\nonumber\\
                                              &=&(T(\rho_1)(X,Y_1),T(\rho_2)(X,Y_2)),\nonumber \\
                                              &=&(\widetilde{\Phi_1} \oplus \widetilde{\Phi_2})(X)(Y_1,Y_2),
\end{eqnarray}
for all $X\in TG$ and $(Y_1,Y_2)\in T(V_1)\oplus T(V_2)$. Therefore $$\widetilde{(\Phi_1 \oplus \Phi_2)}=(\widetilde{\Phi_1} \oplus \widetilde{\Phi_2}).$$ 
\end{proof}

\section{Conclusion} In this study we have prolonged finite-dimensional real  representations of Lie groups. We  have obtained faithful representations of tangent bundles of  Lie groups.
 We have shown that tangent functions of  Lie group actions correspond to prolonged representations. We have proven that if two representations are equivalent then their prolongations are also equivalent. We have shown that if $U$ is an invariant subspace for a representation, then $TU$ is an invariant subspace for prolongation of given representation and vice versa. We have proven that if $\widetilde{\Phi}$ is an irreducible representation, then $\Phi$ is also an irreducible representation. Finally we have shown that prolongations of direct sum of two representations are direct sum of prolongations of them. In future we plan to study on prolongations of representations of Lie algebras.


\begin{thebibliography}{99}
\bibitem{Arv} \textsc{Arvanitoyeorgos, A. }, An Introduction to Lie Groups and the Geometry of Homogeneous Space, AMS, Student Mathematical Library, (2003)
\bibitem{Bel} 
\textsc{Belinfante Johan G. F., Kolman Bernard}, A survey of Lie Groups and Lie Algebras with Applications and Computational Methods, SIAM,(1989).
\bibitem{Clark} 
\textsc{Brickell F., Clark R.S.}, Differentiable Manifolds An Introduction, Van Nostrand Reinhold Company , London,(1970).
\bibitem{Greub} 
\textsc{Greub W., Halperin S., Vanstone R.}, Connections, Curvature and Cohomology,2,  Academic Press, New York and London,(1974).
\bibitem{Hall} 
\textsc{Hall, Brian C.}, Lie Groups, Lie Algebras and Representations, Springer-Verlag, New York, (2004).
\bibitem{Morimoto:1968}
\textsc{Morimoto A.}, Prolongations of G-Structures To Tangent Bundles,Nagoya Math. J., \textbf{12}, 67-108 (1968).   
\bibitem{Saunders} 
\textsc{Saunders D.J.}, The Geometry of Jet Bundles, Cambridge University Press, Cambridge-New York, (1989).
\bibitem{Varadajan} \textsc{Varadajan, V.S.}, Lie Groups, Lie Algebras, and Their Representations, Springer-Verlag New York Inc., (1984)
\bibitem{Warner} 
\textsc{Warner, Frank W.}, Foundations of Differentiable Manifolds and Lie Groups  Springer-Verlag , New York, (2000).
\bibitem{Yano1} 
\textsc{Yano K., Kobayashi S.}, Prolongation of Tensor fields and Connections to Tangent Bundles 1, J. Math. Soc. Japan, \textbf{18}, 194-210, (1966).

\end{thebibliography}
\end{document}